\documentclass[12pt]{amsart}
\usepackage{amsmath, amsthm, amssymb}
\usepackage{fullpage}
\usepackage{color}
\usepackage{hyperref}

\newtheorem{theorem}{Theorem}
\newtheorem{lemma}{Lemma}
\newtheorem{proposition}[lemma]{Proposition}

\numberwithin{lemma}{section}

\newcommand{\Ez}{E_0}

\newcommand{\E}{E}

\newcommand{\M}{\mathfrak M}

\numberwithin{equation}{section}

\newcommand{\tW}{{\tilde W}}
\newcommand{\tQ}{{\tilde Q}}

\renewcommand{\AA}{\mathbf A}
\newcommand{\WH}{{\mathcal{WH} }}

\newcommand{\uu}{{\mathbf u}}

\newcommand{\qq}{{\mathbf q}}
\newcommand{\ww}{{\mathbf w}}
\renewcommand{\ggg}{{\mathbf g}}

\newcommand{\tG}{\tilde{G}}
\newcommand{\tK}{\tilde{K}}

\newcommand{\sgn}{\mathop{\mathrm{sgn}}}

\newcommand{\dH}{{\dot{\mathcal H} }}

\newcommand{\W}{{\mathbf W}}

\renewcommand{\S}{{\mathbf S}}
\newcommand{\err}{\text{\bf err}}

\begin{document}

\title{Two dimensional water waves in holomorphic coordinates II: global solutions}

\author{Mihaela Ifrim}
\address{Department of Mathematics, University of California at Berkeley}
\thanks{The first author was supported by the Simons Foundation}
\email{ifrim@math.berkeley.edu}

\author{ Daniel Tataru}
\address{Department of Mathematics, University of California at Berkeley}
 \thanks{The second author was partially supported by the NSF grant DMS-1266182
as well as by the Simons Foundation}
\email{tataru@math.berkeley.edu}

\begin{abstract}
  This article is concerned with the infinite depth water wave
  equation in two space dimensions.  We consider this problem
  expressed in position-velocity potential holomorphic coordinates,
and prove that small localized data leads to global solutions.
This article is a continuation of authors' earlier paper \cite{HIT}.
\end{abstract}

\maketitle

\section{Introduction}
We consider the two dimensional infinite depth  water wave equations with gravity but without surface tension.  
This is governed by the incompressible Euler's equations with boundary conditions on the
water surface. Under the additional assumption that the flow is
irrotational the fluid dynamics can be expressed in terms of a
one-dimensional evolution of the water surface coupled with the trace
of the velocity potential on the surface. 

This problem was previously considered by many other authors. The
local in time existence and uniqueness of solutions was proved in
\cite{n,y, wu2}, both for finite and infinite depth.  Later,
Wu~\cite{wu} proved almost global existence for small localized data.
Very recently, global results for small localized data were
independently obtained by Alazard-Delort~\cite{ad} and by
Ionescu-Pusateri~\cite{ip}. Extensive work was also done on the same
problem in three or higher space dimensions, and also on related
problems with surface tension, vorticity, finite bottom, etc. Without
being exhaustive, we list some of the more recent references
\cite{abz, abz1, chs, cl,cs, DL, HL, o, sz, zz}.

An essential choice in any approach to this problem is that of the
coordinates used. The citations above largely rely on either Eulerian
or Lagrangian coordinates.  Instead, the present article relies on
holomorphic coordinates, which were originally introduced by
Nalimov~\cite{n}; these are briefly described below.  In the earlier article
\cite{HIT}, using holomorphic coordinates, we revisited this problem
in order to provide a new, self-contained approach, which considerably
simplified and improved on many of the results mentioned above. Our
results included:

(i) local well-posedness in Sobolev spaces, improving on previous
regularity thresholds, e.g. those in \cite{abz}.

(ii) cubic lifespan bounds for small data. These are proved using a \emph{modified energy
method}, first introduced in the authors' previous article \cite{BH}. The idea there 
is that instead of trying to transform the equation using  the normal form method, which 
does not work well in quasilinear settings, one can produce 
quasilinear energy functionals, which are conserved to cubic order.

(iii) almost global well-posedness for small localized data, refining and 
simplifying Wu's approach in \cite{wu}. 

Here we improve the result in (iii) to a global statement,
drastically  improving and simplifying the earlier results of
Alazard-Delort~\cite{ad} and by Ionescu-Pusateri~\cite{ip}.

We first recall the set-up and the equations.  We denote the water
domain at time $t$ by $\Omega(t)$, and the water surface at time $t$ by
$\Gamma(t)$. We think of $\Gamma(t)$ as being asymptotically flat at
infinity.  Rather than working in cartesian coordinates and the
Eulerian setting, we use time dependent coordinates defined via a
conformal map $\mathcal{F}: \mathbb{H} \to \Omega(t)$, where $\mathbb{H}$ is the lower half plane, $\mathbb{H}:=\left\lbrace  \alpha +i\beta\ : \ \beta <0 \right\rbrace $.
 We also have $\mathcal{F}(\mathbf{R}) = \Gamma(t)$.  We call these the holomorphic
coordinates.  

The real variable $\alpha$ is then used to parametrize
the free surface $\Gamma(t)$.  We say that a function of $\alpha$ is
holomorphic if its Fourier transform is supported in $(-\infty,0]$.
They can be described by the relation $Pf = f$,
where the projector operator $P$ to negative frequencies can be
defined using the Hilbert transform $H$ as
\begin{equation*}
P:= \frac12(I-iH).
\end{equation*}

Our variables $(Z,Q)$ are functions of $t$ and $\alpha$ which  
represent the position of the water surface $\Gamma (t)$,
respectively the holomorphic extension of the velocity
potential restricted to $\Gamma(t)$, expressed in the holomorphic coordinates.
In view of our choice of coordinates, it is natural to consider
the evolution of $(Z,Q)$  within the closed subspace of
holomorphic functions within various Sobolev spaces.

In position-velocity potential holomorphic coordinates the equations have the form
\begin{equation*}
\left\{
\begin{aligned}
& Z_t + F Z_\alpha = 0 \\
& Q_t + F Q_\alpha -i (Z-\alpha) + P\left[ \frac{|Q_\alpha|^2}{J}\right]  = 0, \\
\end{aligned}
\right.
\end{equation*}
where
\begin{equation*}
 F := P\left[ \frac{Q_\alpha - \bar Q_\alpha}{J}\right] , \qquad J := |Z_\alpha|^2.
\end{equation*}
For the derivation of the above equations we refer the reader to \cite{HIT},
 \emph{Appendix A}. With the substitution $W := Z-\alpha$
they become
\begin{equation}
\label{ww2d1}
\left\{
\begin{aligned}
& W_t + F (1+W_\alpha) = 0 \\
& Q_t + F Q_\alpha -i W + P\left[ \frac{|Q_\alpha|^2}{J}\right]  = 0, \\
\end{aligned}
\right.
\end{equation}
where
\begin{equation*}
 F = P\left[\frac{Q_\alpha - \bar Q_\alpha}{J}\right], \qquad J = |1+W_\alpha|^2.
 \end{equation*}

We can also differentiate and rewrite the system in terms of the 
diagonal variables 
\[
(\W,R): = \left(W_\alpha, \frac{Q_\alpha}{1+W_\alpha}\right).
\]
This yields the self-contained system
\begin{equation} \label{ww2d-diff}
\left\{
\begin{aligned}
 & \W_{ t} + b \W_{ \alpha} + \frac{(1+\W) R_\alpha}{1+\bar \W}   =  (1+\W)M
\\
& R_t + bR_\alpha = i\left(\frac{\W - a}{1+\W}\right),
\end{aligned}
\right.
\end{equation}
where  the real {\em advection velocity} $b$ is given by
\begin{equation*}
b := P \left[\frac{{Q}_\alpha}{J}\right] +  \bar P\left[\frac{\bar{Q}_\alpha}{J}\right],
\end{equation*}
and  the real {\em frequency-shift} $a$ is 
\begin{equation*}
a := i\left(\bar P \left[\bar{R} R_\alpha\right]- P\left[R\bar{R}_\alpha\right]\right).
\end{equation*}
The auxiliary function $M$ has the expression
\begin{equation*}
M :=  \frac{R_\alpha}{1+\bar \W}  + \frac{\bar R_\alpha}{1+ \W} -  b_\alpha =
\bar P [\bar R Y_\alpha- R_\alpha \bar Y]  + P[R \bar Y_\alpha - \bar R_\alpha Y],
\end{equation*}
written in terms  of  $Y$ given by 
\[
Y := \frac{\W}{1+\W}.
\]

In particular, we remark that the linearization of the system
\eqref{ww2d1} around the zero solution is a dispersive partial differential equation of the form
\begin{equation} \label{ww2d-0} \left\{
\begin{aligned}
 & w_{ t} +  q_\alpha   =  0
\\
& q_t - i w = 0.
\end{aligned}
\right.
\end{equation}

Now we recall the function spaces introduced in \cite{HIT}. The system \eqref{ww2d-0} is a
well-posed linear evolution in the space $\dH_0$ of holomorphic
functions endowed with the $L^2 \times \dot H^{\frac12}$ norm. A
conserved energy for this system is
\begin{equation*}
\Ez (w,r) = \int \frac12 |w|^2 + \frac{1}{2i} (r \bar r_\alpha - \bar r r_\alpha) \, d\alpha.
\end{equation*}
The nonlinear system \eqref{ww2d1} also admits a conserved energy, which has the form
\begin{equation*}
\E(W,Q) = \int \frac12 |W|^2 + \frac1{2i} (Q \bar Q_\alpha - \bar Q Q_\alpha)
- \frac{1}{4} (\bar W^2 W_\alpha + W^2 \bar W_\alpha)\, d\alpha.
\end{equation*}

 As suggested by the above energy, the function spaces for the 
 differentiated water wave system \eqref{ww2d-diff} are the
spaces $\dH_n$ endowed with the norm
\[
\| (\W,R) \|_{\dH_n}^2 := \sum_{k=0}^n
\| \partial^k_\alpha (\W,R)\|_{ L^2 \times \dot H^\frac12}^2,
\]
where $n \geq 1$. 

To describe the lifespan of the solutions we define the control norms
\begin{equation*}
A := \|\W\|_{L^\infty}+\| Y\|_{L^\infty} + \|D^\frac12 R\|_{L^\infty \cap B^{0,\infty}_{2}},
\end{equation*}
respectively
\begin{equation*}
B :=\|D^\frac12 \W\|_{BMO} + \| R_\alpha\|_{BMO}.
\end{equation*}
Here $A$ is a scale invariant quantity related to the critical 
homogeneous $\dH_\frac12$ norm of $(\W, R)$,  while $B$ corresponds to the
homogeneous $\dH_1$ norm of $(\W, R)$. We note that $B$ and all but
the $Y$ component of $A$ are directly controlled by the $\dH_1$ norm of the
solution.  

The main local well-posedness result in \cite{HIT} is
\begin{theorem}
\label{baiatul}
 Let  $ n \geq 1$.  The system \eqref{ww2d-diff} is locally well-posed
for data in $\dH_n(\mathbb{R})$ so that $|\W+1| > c > 0$.
Further, the solution can be continued for as long as $A$ and $B$ 
remain bounded. 
\end{theorem}

To state the global result we need to return to the original set of
variables $(W,Q)$.  We also take advantage of the scale invariance of
the water wave equations. Precisely, it is invariant with respect to
the scaling law
\[
(W(t,\alpha), Q(t,\alpha)) \to (\lambda^{-2} W(\lambda t,\lambda^2 \alpha),
 \lambda^{-3} Q(\lambda t,\lambda^2 \alpha)).
\]
This suggests that we should use the scaling vector field
\[
S = t \partial_t + 2 \alpha \partial_\alpha,
\]
and its action on the pair $(W,Q)$, namely
\[
\S(W,Q) := ((S-2)W,(S-3)Q),
\]
which solve the linearized equations, see \cite{HIT}.
However, these are not the correct diagonal variables;
instead the diagonal variables are $\AA\S(W,Q)$,
where the diagonalization operator $\AA$ is given by 
\[
 \AA(w,q) := (w, q-Rw).
\]
Then we  define the weighted energy
\begin{equation*}
\|(W,Q)(t)\|_{\WH}^2 :=  \|(W,Q)(t)\|_{\dH_0}^2 + \|(\W,R)(t)\|_{\dH_5}^2 
+ \|\AA\S(W,Q)(t)\|_{\dH_1}^2.
\end{equation*}
To control the evolution of the weighted energy we still use 
a pointwise type control norm, but one which is somewhat stronger 
than $A$ and $B$. Precisely, we define
\begin{equation*}
\| (W,R)\|_{X} := \|W\|_{L^\infty} + \|R\|_{L^\infty} +  \|D^2 W\|_{L^\infty} + \| |D|^\frac32 R\|_{L^\infty}.
\end{equation*}

Now we can state our main result:

\begin{theorem} \label{t:global}
a) (Global solutions) Let  $\epsilon \ll 1$. Then for each initial data
$(W(0),Q(0))$ for the system \eqref{ww2d1} satisfying
\begin{equation}\label{data}
\|(W,Q)(0)\|_{\WH}^2  \leq \epsilon,
\end{equation}
the solution is global, 
and satisfies
\begin{equation}\label{energy}
\|(W,Q)(t)\|_{\WH}^2  \lesssim \epsilon t^{C\epsilon^2},
\end{equation}
as well as 
\begin{equation}\label{point}
\|(W,R)\|_{X}  \lesssim \frac{\epsilon}{\sqrt{t}}.
\end{equation}

b) (Asymptotic profile) There exists a function $\Psi$ satisfying
\begin{equation*}
\| (1+v^{-2})^{-5}  \Psi\|_{L^2} + \| v \partial_v \Psi\|_{L^2} \lesssim \epsilon ,
\end{equation*}
so that we have the asymptotic formulas
\begin{equation}
\label{asymptotics}
\begin{split}
(W,Q)(t,\alpha)  = & \ \frac{1}{\sqrt{t}} e^{i\frac{t^2}{4\alpha}} \Psi(\alpha/t)
e^{\frac{i}2 (2\alpha/t)^{-5} \ln{t} |\Psi(\alpha/t)|^2}\left(1,\frac{t}{2\alpha}\right) + e_\alpha(t,\alpha/t),
\\
(\hat W,\hat Q)(t,\xi)  = &  \sum_{-4v^2 = \xi}  |v|^{-\frac32}
 e^{\pm  i t  \sqrt{|\xi|}} \left[ \Psi\left(v \right)
e^{\frac{i}2 (2v)^{-5} \ln{t} |\Psi(v)|^2} (1,(2v)^{-1}) + e_\xi(t,\xi)\right] ,
\end{split}
\end{equation}
where the last sum has two terms, depending on the sign of $v$, 
and the errors $e_\alpha$ and $e_\xi$, satisfy bounds of the form 
\begin{equation}
\label{asy-err}
\begin{split}
 e_\alpha(t,v) =  \ O_{ L^\infty}(\epsilon t^{-\frac{5}9}), \quad 
e_\xi(t,\xi) = & \ O_{ L^\infty}(\epsilon t^{-\frac1{18}}).
\end{split}
\end{equation}

\end{theorem}

We recall that results of these type were recently proved in work of
 Alazard-Delort~\cite{ad} and Ionescu-Pusateri~\cite{ip}. 
Our result here, based on the setup in the previous article \cite{HIT},
provides a stronger statement and a much simpler proof. The main
idea of the proof is described in the simpler setting of the one
dimensional cubic NLS in the companion article \cite{IT}. However,
this article does not rely directly on any of the results proved in
\cite{IT}.  While not needed for the proof of the global well-posedness result,
the modified scattering type asymptotic profile in part (b) is easily obtained as
a direct byproduct of our proof. A similar asymptotic profile, but in Eulerian coordinates, was obtained by Alazard and Delort in \cite{ad}.

The organization of the paper is as follows. In the next section we set-up the 
bootstrap argument for the proof of part (a) of the theorem. The energy estimates 
were already proved in \cite{HIT}, and they are simply recalled here. In preparation 
for the proof of the pointwise bounds we also recall the normal form transformation 
associated to the two dimensional water wave equation, along with some related 
energy bounds. The core of the paper is Section~\ref{s:section3}, where we prove the global pointwise
bounds and close the bootstrap argument. Our approach is based on the idea 
of testing with wave packets, which was first developed in \cite{IT}.
The asymptotic profile is readily obtained at the end of the proof.

\section{The bootstrap setup}
\label{s:section2}
The main difficulty in the proof of the theorem is in establishing the
sharp pointwise decay rate in \eqref{point}. Since in our context
$\|(W,R)(t)\|_{X}$ is a continuous function of $t$, without any loss of generality ,
we can make the bootstrap assumption
\begin{equation}\label{point-boot}
 \|(W,R)(t)\|_{X} \leq \frac{C \epsilon}{\sqrt{t}}
\end{equation}
with a fixed large universal  constant $C$.  Here $C$ is chosen with the property that 
\[
1 \ll C \ll \epsilon^{-\frac12}.
\]
Then  we need to show 
that for any solution satisfying \eqref{data} and \eqref{point-boot} in a time interval $[0,T]$,
we must also have \eqref{energy} and \eqref{point}.

The  energy estimates are nontrivial, but they were already established in \cite{HIT}.
Precisely, from Proposition~6.1 in \cite{HIT} we have

\begin{proposition}
Assume that \eqref{data} and \eqref{point-boot} hold  in a time interval $[0,T]$.
Then we also have the energy estimate
\begin{equation}
\label{e}
  \|(W,Q)(t)\|_{\WH}^2  \lesssim \epsilon t^{C_*^2 \epsilon^2}, \qquad t \in [0,T], \qquad 
C_* \lesssim C.
\end{equation}
\end{proposition}
Thus, the bound \eqref{energy} follows directly from our bootstrap assumption.
A pointwise bound is also established in \cite{HIT}, see Proposition~6.2:
\begin{proposition}
Assume that \eqref{e} holds in a time interval $[0,T]$. Then we  also  have 
 the pointwise bounds
\begin{equation}\label{first-point}
|W|+|W_\alpha|+|W_{\alpha \alpha}| + |D^\frac12 Q| + |R| + |R_\alpha|+|D^\frac12 R_\alpha|
\lesssim  \epsilon t^{C_*^2 \epsilon^2} \omega(\alpha,t),
\end{equation}
where 
\[
\omega(\alpha,t) := \frac{1}{t^\frac{1}{18}} + \frac{1}{(|\alpha|/t + t/|\alpha|)^\frac12}.
\]
\end{proposition}
We remark that this proposition gives the bound 
\begin{equation}\label{first-pointX}
\|(W,R)\|_{X}  \lesssim  \epsilon t^{C _{*}^2\epsilon^2} t^{-\frac12},
\end{equation}
which is only sufficient in order to close the bootstrap up to an
exponential time $T$, $T \lesssim e^{-c \epsilon^{-2}}$, and thus prove 
the almost global result.  The bound 
\eqref{first-pointX} will not be so useful to us since we already have the bootstrap
assumption \eqref{point-boot}. However, we can get more use out of
\eqref{first-point}; what \eqref{first-point} shows is 
that there is extra decay away from $|\alpha| \approx t$, so it suffices to
improve the pointwise bound in \eqref{first-point} in a region of the
form
\begin{equation}
\Omega = \left\{ t^{-\frac19} \lesssim \frac{|\alpha|}{t} \lesssim t^\frac19 \right\}.
\end{equation}
Here the threshold $\frac19$ was chosen somewhat arbitrarily; any smaller power
would work as well.

The goal of the remainder of the paper is to establish the 
bound \eqref{point} in $\Omega$. A key tool in this endeavor is the 
normal form transformation, which is discussed next.

\subsection{The normal form transformation}

For the almost global result in \cite{HIT}, as well as for the global
result here, a very useful observation is that the quadratically
nonlinear terms may be removed from the water-wave equations by the
near-identity, normal form transformation
\begin{equation}
\tilde W = W - 2 \M_{\Re W} W_\alpha, \qquad \tilde Q = Q - 2 \M_{\Re W} R,
\label{nft1}
\end{equation}
where the holomorphic multiplication operator $\M_f$ is given by $\M_f
g = P\left[ fg\right] $.  For a more symmetric form of this
transformation, one can replace $R$ by $Q_{\alpha}$. However, it is
more convenient to use the diagonal variable $R$.  The goal of normal form
transformation is to remove the quadratic terms in the equation.
Precisely, we have
\begin{proposition}
The normal form variables solve a cubic equation:
\begin{equation}\label{cubic}
\left\{
\begin{aligned}
&\tW_t + \tQ_\alpha = \tG
\\
&\tQ_t - i \tW =  \tK,
\end{aligned}
\right.
\end{equation}
where $\tG$, $\tK$ are cubic (and higher order) functions of
$(W, \W,R, \W_{\alpha}, R_{\alpha})$, given by
\begin{equation}\label{gk-tilde}
\left\{
\begin{aligned}
\tilde G =  &\  2P[ (F - R)_\alpha \Re W + \W_{ \alpha} F\Re W + \W \Re (\W F)+F_{\alpha }\W\Re W]
\\ & \ -  P[\bar \W R \bar Y - \W(P[\bar R Y]+\bar P[R \bar Y])]
\\
\tilde K = & \  P\left[ ((1+\bar \W)\bar F- \bar R)R + 2P [bR_\alpha]\Re W   +     2i P \left[ \frac{\W^2 + a}{1+\W}\right]\Re W \right] .
\end{aligned}
\right.
\end{equation}
\end{proposition}

We remark here that the normal form transformation cannot be used directly 
to study well-posedness questions for the water wave equation as the cubic and higher
order terms on the right are higher order than the leading linear part. However, 
in \cite{HIT} we were able to use it in order to derive the pointwise bounds
in \eqref{first-point}, and here we will be able to further use it to get to \eqref{point}.

In order to work with $(\tW,\tQ)$ instead of $(W,Q)$ we need to be able to 
transfer the energy information from $(W,Q)$ to $(\tW,\tQ)$, and the 
pointwise bounds in the opposite direction. This was also done in \cite{HIT}:

\begin{proposition}
Assume that the energy bound \eqref{e}  holds in a time interval $[0,T]$. 
Then we have the following estimates for $(\tW,\tQ)$:  

(i) Energy estimates:
\begin{equation}\label{te}
\|(\tW,\tQ)\|_{\dH_5} \lesssim \epsilon t^{C_{*}^2\epsilon^2},
\end{equation}
\begin{equation}\label{tvf}
\| (2 \alpha \partial_\alpha \tilde W +  t \partial_\alpha \tilde Q ,  
2\alpha \partial_\alpha \tilde Q - it \tilde W )\|_{\dH_0} \lesssim \epsilon t^{C_{*}^2 \epsilon^2}.
\end{equation}

(ii) Pointwise comparison:
\begin{equation}\label{point-comp}
\| (W-\tW, R - \tQ_\alpha)\|_{X} 
\lesssim  \epsilon^2 t^{-\frac58+2C_{*} \epsilon^2} .
\end{equation}
\end{proposition}

For the proof of this result we refer the reader to the corresponding
results in \cite{HIT} as follows.  The energy estimates in \eqref{te}
are contained in Lemma~6.4. The bound \eqref{tvf} is based on the
computation in (6.13), so it requires both Lemma~6.4 and
Lemma~6.5. The estimate \eqref{point-comp} is a consequence of
Lemma~6.3, and is used in \cite{HIT} to prove that the bound \eqref{first-point}
for $(W,R)$ is equivalent to its counterpart for  $(\tW,\tQ_\alpha)$, namely 
\begin{equation}\label{first-point-tilde}
|\tW|+|\tW_{\alpha \alpha}| + ||D|^\frac12 \tQ| + |\tQ| + |R_\alpha|+||D|^\frac12 \tQ_{\alpha\alpha}|
\lesssim  \epsilon t^{C_*^2 \epsilon^2} \omega(\alpha,t).
\end{equation}

We remark that in view of \eqref{point-comp}, for the proof of
\eqref{point}, it suffices to obtain the uniform bounds associated to
$(\tW,\tQ)$,
\begin{equation}\label{point-t}
|\tW|+|\tW_{\alpha \alpha}| + ||D|^\frac12 \tQ| +||D|^\frac12 \tQ_{\alpha\alpha}|
\lesssim  \epsilon t^{-\frac12}
\end{equation}
in the region $\Omega$.

Next we consider the right hand side terms $\tG$ and $\tK$ in the
equations for $(\tW,\tQ)$.  It is convenient to decompose them into
cubic and higher terms. 
\[
\tG = \tG^{(3)} + \tG^{(4+)}, \qquad \tK = \tK^{(3)} + \tK^{(4+)},
\]
where 
\begin{equation}
\label{gk-tilde1}
\left\{
\begin{aligned}
\tilde G^{(3)} = & \  2P[ -\partial_{\alpha}P\left[ R\bar{ \W}-\bar{R}\W\right] \Re W +(R \W)_{ \alpha}  \Re W + \W \Re (\W R)]   \\ & \ -  P[\bar \W^2  R- \W(P[\bar R \W]+\bar P[R \bar \W]) ]\\
\tilde K^{(3)} = &\  2P\left[ P [(R+\bar R)R_\alpha]\Re W   +     i \W^2 \Re W + P[R \bar R_\alpha] \Re W \right] \\ & \ +P\left[\bar{R}\bar{\W}R -R\bar{P}\left[ \bar{R}\W-R\bar{\W} \right]\right].
\end{aligned}
\right.
\end{equation}
The higher order terms play a perturbative role in the long time
behavior, because of the better decay. However, the cubic terms may
drive the asymptotic dynamics, and need to be considered in greater
detail.  For this reason, it is also convenient to express the cubic
terms as cubic in $(\tW,\tQ_\alpha)$ rather than $(W,R)$.  For the
perturbative part of the terms above we have the following

\begin{proposition}
Assume that the bound \eqref{e} holds. Then we have
\begin{equation}\label{GK4}
\| (\tilde G^{4+},  \tilde K^{4+})\|_{\dH_0} \lesssim \epsilon^4 t^{-\frac32+3C_{*}^2 \epsilon^2},
\end{equation}
respectively
\begin{equation}
\label{GK-diff}
\| (\tilde G^{3}(W,R),  \tilde K^{3}(W,R)) - (\tilde G^{3}(\tW,\tQ_\alpha),  \tilde K^{3}(\tW,\tQ_\alpha))
\|_{\dH_0} \lesssim \epsilon^4 t^{-\frac32+3C_{*}^2 \epsilon^2}.
\end{equation}
\end{proposition}
\begin{proof}
The first estimate \eqref{GK4} follows from  by interpolation and H\"older's
inequality from the bounds
\[
\| (W_\alpha,R)\|_{\dH^4} \lesssim \epsilon t^{C_{*}^2 \epsilon^2},
\]
\[
|W|+|W_\alpha|+|W_{\alpha\alpha}| + |R|+|R_\alpha|  \lesssim \epsilon t^{-\frac12+C_{*}^2 \epsilon^2},
\]
which in turn are consequences of \eqref{e} and
\eqref{first-point}. For the second bound \eqref{GK-diff} we also need
to use once the estimate \eqref{point-comp}. The details are somewhat
tedious but routine, and are left for the reader.
\end{proof}

\section{ Pointwise decay}
\label{s:section3}

\subsection{Testing by packets}

In order to establish the global pointwise decay estimates we use the
method of testing by wave packets, first introduced in the companion
paper \cite{IT} in the context of the one dimensional cubic NLS
equation.  The procedure we apply is very simple; we pick a ray
$\{\alpha = v t\}$ and establish decay along this ray by testing with a
wave packet moving along the ray. A wave packet is an approximate
solution to the linear system \eqref{ww2d-0}, with $O(1/t)$ errors.

To motivate the definition of this packet we recall some
useful facts. In view of the dispersion relation $\tau = \pm \sqrt{|\xi|}$,
a ray with velocity $v$ is associated with waves which have 
spatial frequency 
\[
\xi_v = - \frac{1}{4 v^2}.
\]
Secondly, for waves with initial data localized at the origin,
the spatial frequency corresponding with a position $(\alpha,t)$ is 
\[
\xi(\alpha,t) = -\frac{t^2}{4\alpha^2}.
\]
This is associated with the phase function
\[
\phi(t,\alpha) = \frac{t^2}{4\alpha}.
\]

Then our wave  packets will be combinations of functions of the 
\[
\uu(t,\alpha) = v^{-\frac32} \chi\left(\frac{\alpha - vt}{t^\frac12 v^\frac32}\right)e^{i\phi(t,\alpha)},
\]
where $\chi$ is a smooth compactly supported bump function with integral one
\begin{equation}\label{chi-int}
\int \chi (y)\, dy =1.
\end{equation}
Our packets are localized around the ray $\{\alpha = v t\}$ on the
scale $\delta \alpha = t^\frac12 v^\frac32$. This exact choice of
scale is determined by the phase function $\phi$. Precisely, the quadratic 
expansion of $\phi$ near $\alpha = vt$ reads
\[
\phi(t,\alpha) = \phi(t,vt) + (\alpha-vt) \phi_\alpha(t,vt) + O( t^{-1} v^{-3}  (\alpha-vt)^2),
\]
and our scale $\delta \alpha$ represents exactly the scale on which $\phi$ is well
approximated by its linearization. We further remark that there is a threshold
$v \approx t$ above which $\phi$ is essentially zero, and the above considerations
are no longer relevant. We confine our analysis to the region where $\phi$ is strongly
oscillatory, 
\begin{equation*}
|v| \ll t^{\frac12}.
\end{equation*}
The power $\frac12$ here is somewhat arbitrary, any choice less than $1$ would do. 
Under this assumption, the function $\uu$ is strongly localized at frequency $\xi_v$.
For later use, we record here some ways to express this localization.

\begin{lemma}
\label{l:uu}
a) Let $\uu$ be defined as above. Then its Fourier transform and that
of $\partial_v \uu$ have the form
\begin{equation}\label{hatq}
\hat \uu (\xi)=  t^\frac12 \chi_1\left(\frac{\xi + (4v^2)^{-1}}{ t^{-\frac12} v^{-\frac32}} \right) e^{- i t \sqrt{|\xi|}},
\qquad \partial_v \hat{\uu}(\xi) =  t v^{-\frac32}  \chi_2\left(\frac{\xi + (4v^2)^{-1}}{ t^{-\frac12} v^{-\frac32}} \right) e^{- i t \sqrt{|\xi|}} ,
\end{equation}
 where $\chi_1$ and $\chi_2$ are Schwartz functions so that in addition, 
\begin{equation}\label{chi1-int}
\int \chi_1(\xi) \, d\xi = 1 + O(v^{\frac12}t^{-\frac12}).
\end{equation}

b) For $s \geq 0$, $\lambda_v = (4v^2)^{-1}$ and $P_{\lambda_v}$ the associated dyadic frequency projector  we have
\begin{equation} \label{dsuu}
P_{\lambda_v}(|D|^s   - (4v^2)^{-s}) \uu(\alpha,t)  =  (4v^2)^{-s} t^{-\frac12} v^{\frac12}
\chi_3\left(\frac{\alpha - vt}{t^\frac12 v^\frac32}\right)e^{i\phi(t,\alpha)},
\end{equation}
where $\chi_3$ is also a Schwartz function.
\end{lemma}

The proof of the lemma is straightforward, and left for the
reader.  In order to obtain\eqref{chi1-int}, the key idea is to replace phases by their quadratic
approximations; see also the similar computation in \cite{IT}.

\bigskip

Applying the method of testing by wave packets for the water wave
equation is slightly more complicated than in the case of the cubic
NLS in \cite{IT} due the fact that we are dealing with a system, and
we need to choose the two components to match.  However, our system is
simple enough, so is suffices to first choose the $Q$ component and
then use the second of the two linear equations in \eqref{ww2d-0}  to match $W$,
\begin{equation*}
(\ww,\qq) = ( -i v \partial_t \uu, v \uu).  
\end{equation*}
Then we have
\begin{equation}
\label{w-bold}
\ww = \frac12 \uu + \left( \frac{vt -\alpha}{2 \alpha} \chi\left(\frac{\alpha - vt}{t^\frac12 v^\frac32}\right)
 + \frac{i (vt+\alpha)}{2t^\frac32 v^\frac12}  \chi'\left(\frac{\alpha - vt}{t^\frac12 v^\frac32}\right)\right) v^{-\frac32} e^{i\phi(t,\alpha)}.
\end{equation}
The second term above is better by a $v^\frac12 t^{-\frac12}$ factor, so it will play a negligible role in 
most of our analysis. However, it is crucial in improving the error in the first linear equation
in  \eqref{ww2d-0}, which is given by 
\begin{equation}
\label{g-eq}
\ggg  := \partial_t \ww + \partial_\alpha \qq = v (\partial_\alpha - i \partial_t^2) \uu .
\end{equation}
 Indeed,  computing the  error in \eqref{g-eq} we obtain
\begin{equation}
\label{erori}
\begin{aligned}
(\partial_\alpha - i \partial_t^2) \uu &=\frac{e^{i\phi}}{v^{\frac{3}{2}}} \partial_{\alpha}\left[ \frac{(\alpha-vt) }{2\alpha}\chi-i \frac{ (\alpha +vt)^2}{4v^{\frac{3}{2}}t^{\frac{5}{2}}}\chi '\right]+ \frac{e^{i\phi}}{v^{\frac{3}{2}}} \left[ \frac{(\alpha-vt)}{2\alpha^2}\chi -i \frac{ (\alpha -vt)}{4v^{\frac{3}{2}}t^{\frac{5}{2}}}\chi'\right].
\end{aligned}
\end{equation}
The first term is the leading one, and, as expected, has size $t^{-1}$
times the size of $\ww$; further, it exhibits some additional
structure, manifested in the presence of $\partial_\alpha$, which we
will take advantage of later on.  The terms in the second bracket of
the RHS are better by another $t^{\frac12}$ factor, so no further structure
information is needed.

This shows that the choice of a such wave packet is a reasonable approximate
solution for the solution the the linear system.  Precisely, as in \cite{IT},
our test packets $(\ww,\qq)$ are good approximate solutions for the linear
system associated to our problem only on the dyadic time scale $\delta
t\leq t$.  

The outcome of testing the normal form solutions to the water wave
system with the wave packet $(\ww,\qq)$ is  the
scalar complex valued function $\gamma(t,v)$, defined in the region $\{|v| \leq t^\frac12\}$:
\begin{equation*}
\gamma(t,v) = \langle (\tW,\tQ),(\ww,\qq)\rangle_{\dH_0},
\end{equation*}
which we will use as a good measure of the size of $(\tW,\tQ)$ along
our chosen ray.  Here it is important that we use the complex pairing
in the inner product. 

Now we have two
tasks. Firstly, we need to show that $\gamma$ is a good representation of the 
pointwise size of $(\tW,\tQ)$ and their derivatives:

\begin{proposition}
\label{p:diff}
Assume that \eqref{te} and \eqref{tvf} hold. Then in  $\Omega$ we have 
the following bounds for $\gamma$:
\begin{equation}\label{gamma}
\|(1+ v^{-2})^{5}\gamma \|_{L^2_v}  + \|v \partial_v\gamma \|_{L^2_v} 
+  \|v^\frac12 (1+ v^{-2})^{\frac52}\gamma \|_{L^\infty}  \lesssim 
\epsilon t^{C^2_{*}\epsilon^2},
\end{equation}
as well as the approximation bounds for $(\tW,\tQ)$ and their derivatives:
\begin{equation} \label{pactest}
\begin{split}
&(|D|^s \tW,|D|^{s+\frac12} \tQ) (t,vt) =  \ |\xi_v|^{s}
t^{-\frac12} e^{i\phi(t,vt)} \gamma(t,v) ( 1, \sgn{v}) + \err_s, 
\\  
& (\hat \tW,|\xi|^{\frac12}\hat \tQ)(t,\xi) =   
\sum_{\xi = -(4v^2)^{-1}}  |v|^{-\frac32} (e^{i t \sgn{v}  |\xi|^\frac12}  \gamma(t,v)(1,\sgn{v}) + \hat{\err}),
\end{split}
\end{equation}
where
\begin{equation}\label{pacerr}
\begin{split}
&\|(1+v^{-2})^{2-s} \err_{s+\frac12}\|_{L^2_v}  + \|(1+v^{-2})^{2-s} \err_s\|_{L^\infty} \lesssim  
\epsilon t^{-\frac59+C_{*}^2\epsilon^2}, \qquad 0 \leq s \leq 2,
\\
&\|v^{-1} (1+v^{-2})^2 \hat \err\|_{L^2_v}  + \| (1+v^{-2})^2 \hat \err\|_{L^\infty} \lesssim  
\epsilon t^{-\frac{1}{18}+C^2_{*}\epsilon^2}.
\end{split}
\end{equation}
\end{proposition}

Secondly, we need to show that $\gamma$ stays bounded,  which we do by 
establishing a differential equation for it:

\begin{proposition}
\label{p:ode}
Assume that \eqref{e}, \eqref{first-point}, \eqref{te}, \eqref{tvf} and \eqref{point-comp} hold.
Then within the set $\Omega$ the function $\gamma$ solves an asymptotic ordinary differential equation of the form 
\begin{equation}\label{gamma-ode}
\dot\gamma = \frac{i}{2t (2v)^5}   \gamma |\gamma|^2 + \sigma, 
\end{equation}
where $\sigma$ satisfies the $L^2$ and $L^\infty$ bounds
\begin{equation}\label{sigma}
 \| (1+v^{-2})^{2} \sigma\|_{L^2}+ \| (1+v^{-2})^{2} \sigma\|_{L^\infty}  \lesssim \epsilon
  t^{-\frac{19}{18}+C^2_{*}\epsilon^2}.
\end{equation}
\end{proposition}

We now use the two propositions to conclude the proof of \eqref{point-t}. 
By virtue of \eqref{pactest} and \eqref{pacerr}, in order to prove \eqref{point-t}
it suffices to establish its analogue for $\gamma$, namely 
\begin{equation}\label{need}
|\gamma(t,v)| \lesssim \epsilon (1+v^{-2})^{-2} \qquad \text{  in  $\Omega$}.
\end{equation}
On the other hand,  from \eqref{first-point-tilde} we directly obtain
\begin{equation}\label{have}
|\gamma(t,v)| \lesssim \epsilon (1+v^{-2})^{-2} \omega(v,t) t^{C_*^2 \epsilon} \text{ in $\Omega$}.
\end{equation}

Our goal now is to use the ode \eqref{gamma-ode} in order to transition from 
\eqref{have} to \eqref{need} along rays $\alpha = vt$. We consider three cases for $v$:

(i)  Suppose first that
$v \approx 1$, i.e., $|\alpha| \approx t$. Then we initially have 
\[
|\gamma(t)| \lesssim \epsilon, \qquad t \approx 1.
 \]
Integrating \eqref{gamma-ode} we conclude that 
\[
|\gamma(t)| \lesssim \epsilon, \qquad t \geq 1,
\]
and then \eqref{need} follows.

(ii) Assume now that $v \ll 1$, i.e., $|\alpha| \ll t$. Then, as $t$ increases, the ray $\alpha = vt$
enters $\Omega$ at some point $t_0$ with $v \approx t_0^{-\frac19}$. Then by 
 \eqref{have} we obtain
\[
|\gamma(t_0,v)| \lesssim \epsilon v^4 v^{\frac12} t^{C_*^2 \epsilon}  \lesssim \epsilon v^2.
\] 
We use this to initialize $\gamma$. For larger $t$ we use \eqref{gamma-ode}
to conclude that 
\[
|\gamma(t)| \lesssim \epsilon v^4 + \int_{t_0}^\infty \epsilon v^4 
s^{-\frac{19}{18}+C^2_{*}\epsilon^2} ds \approx 
 \epsilon v^4 (1+ t_0^{-\frac1{18}+C^2_{*}\epsilon^2 }) \lesssim \epsilon v^4 , \qquad t > t_0.
\]
Then  \eqref{need} follows.

(iii) Finally, consider the case $v \gg 1$, i.e., $|\alpha| \gg t$.
Again, as $t$ increases, the ray $\alpha = vt$ enters $\Omega$ at some point $t_0$ 
 $v \approx t_0^{\frac19}$, therefore by \eqref{have} we obtain
\[
|\gamma(t_0,v)| \lesssim \epsilon  v^{-\frac12} t_0^{C_*^2 \epsilon}  \lesssim \epsilon. 
\] 
We use this to initialize $\gamma$. For larger $t$ we use \eqref{gamma-ode}
to conclude that 
\[
|\gamma(t)| \lesssim \epsilon  + \int_{t_0}^\infty \epsilon s^{-\frac{19}{18}+C^2_{*}\epsilon^2}  \approx 
 \epsilon  (1+ t_0^{-\frac1{18}+C^2_{*}\epsilon^2 }) \lesssim \epsilon ,  \qquad t > t_0.
\]
Then  \eqref{need} again follows.

The remainder of the paper is devoted to the proof of the two propositions above.

\subsection{Approximation errors.}

Here we prove Proposition~\ref{p:diff}, using the estimates \eqref{te}
and \eqref{tvf}. In order to symmetrize the problem it is useful to introduce the normalized
variables $(w,r) = (\tW, D^\frac12 \tQ)$, which satisfy the bounds
\begin{equation*}
\| (w,r)\|_{H^5} \leq \epsilon t^{C^2_{*} \epsilon^2}, \qquad
\| (2 \alpha \partial_\alpha w -  i t  |D|^\frac12 r ,  
2\alpha \partial_\alpha  r - it |D|^\frac12 w )\|_{L^2_{\alpha}} \lesssim \epsilon t^{C^2_{*} \epsilon^2}.
\end{equation*}
Then we rewrite $\gamma$ in terms of these variables
as
\[
\gamma = \int w \bar \ww + r D^\frac12 \bar \qq \, d\alpha .
\]

For the purpose of proving \eqref{pacerr} we can simplify somewhat the
expression of $\gamma$. The lower order terms in $\ww$ in
\eqref{w-bold} are better by a factor of $v^\frac12 t^{-\frac12}$,
therefore we can readily replace $\ww$ by $\frac12 \uu$, modulo errors
which satisfy \eqref{pacerr}. Also, in view of Lemma~\ref{l:uu}, we
can also substitute $D^\frac12 \bar \qq$ by $ \frac{t}{2|\alpha|} \bar
\qq$ and further by $\pm \frac12\uu $, with errors that are also
$v^\frac12 t^{-\frac12}$ better. In view of these considerations, it suffices to
prove Proposition~\ref{p:diff} with $\gamma$ redefined as
\begin{equation}
\label{rewrite-gamma}
\gamma(t,v) = \frac{1}{2} \int (w\pm r) \bar{\textbf{q}} \, d\alpha.
\end{equation}
Then Proposition~\ref{p:diff}  is a consequence of the following Lemma:

\begin{lemma}
Let $\gamma$ be defined as in \eqref{rewrite-gamma} in the region 
$\Omega$, where $(w,r)$ are holomorphic functions
which satisfy
\begin{equation}
\label{H}
\| (w,r)\|_{H^5} \leq 1, \qquad \| (2 \alpha \partial_\alpha w -  i t  |D|^\frac12 r ,  2\alpha \partial_\alpha  r - it |D|^\frac12 w )\|_{L^2_{\alpha}} \lesssim 1.
\end{equation}
Then $\gamma$ satisfies the bounds
\begin{equation}
\label{point-gamma}
 \|(1+v^{-2})^5 \gamma\|_{L^2_v}  + \| v \partial_v \gamma \|_{L^2_v} \lesssim  1, 
\qquad |\gamma| \lesssim v^{-\frac12}(1+v^{-2})^{-\frac52}.
\end{equation}
Moreover, the following error bounds for $\gamma$ also hold:
\begin{equation}\label{pactest1}
\begin{split}
&|D|^s(w,r)(t,vt) =   |\xi_v|^s t^{-\frac12}e^{i\phi(t,vt)} \gamma(t,v)(1,\sgn{v}) + \err_s, \\
&  (\hat w,\hat r)(t,\xi) =  \ \sum_{\xi = - (4v^2)^{-1}}  |v|^{-\frac32}(  
e^{i t \sgn{v}  |\xi_v|^\frac12}  \gamma(t,v)(1,\sgn{v}) + \hat{\err}),
\end{split}
\end{equation}
where
\begin{equation}\label{pacerr1}
\begin{split}
&\|(1+v^{-2})^{2-s} \err_{s+\frac12} \|_{L^2_v} \lesssim  t^{-\frac59}, \qquad  \|(1+v^{-2})^{2-s} \err_s\|_{L^\infty} \lesssim  t^{-\frac59} , 
\qquad 0 \leq s \leq 2, \\
&\|v^{-1} (1+v^{-2})^{2} \hat\err \|_{L^2_v} \lesssim  t^{-\frac1{18}}, \ \, \qquad  \|(1+v^{-2})^{2} \hat\err\|_{L^\infty} \lesssim  t^{-\frac1{18}} , 
\quad \qquad 0 \leq s \leq 2. 
\end{split}
\end{equation}
\end{lemma}

\begin{proof}
  We first note that our hypothesis \eqref{H} on $(w,r)$ is stable
  with respect to dyadic frequency localizations. At a fixed dyadic frequency $\lambda$,
the operator 
\[
(w,r) \rightarrow  (2 \alpha \partial_\alpha w -  i t  |D|^\frac12 r ,  2\alpha \partial_\alpha  r - it |D|^\frac12 w )
\]
is elliptic outside the region $\lambda \approx t^2/\alpha^2$.  Equivalently,
the spatial region $\{\alpha \approx vt\}$ is matched to the dyadic
frequency $\lambda_v = v^{-2}$.    Thus, using elliptic
estimates, we can decompose $(w,r)$ into a leading part and an elliptic component,
\[
(w,r) = (w_{ell},r_{ell}) + \sum_{v} \chi_{\alpha \approx vt} P_{\lambda_v}(w,r).
\]
We consider the two parts separately.

\bigskip

{\bf (i) The elliptic part:} For   $(w_{ell},r_{ell})$, the bound \eqref{H} translates into 
\[
\| (w_{ell},r_{ell})\|_{H^5} \leq 1, \qquad \|  \alpha \partial_\alpha (w_{ell},r_{ell}) \|_{L^2}
+\| t |D|^\frac12 (w_{ell},r_{ell})\|_{L^2} 
 \lesssim 1 .
\]
By interpolation and Bernstein's inequality this leads to the  bounds
\[
\| |D|^{s+\frac12}  (w_{ell},r_{ell})\|_{L^2_\alpha}+ 
\| |D|^{s}  (w_{ell},r_{ell})\|_{L^\infty} 
\lesssim t^{-\frac59-\frac29 (2-s)}, \qquad 0 \leq s \leq 2.
\]
One can also switch to an $L^2_v$ norm, using the fact
that the norms $L^2_v$ and $L^2_{\alpha}$ are related by
\[
\Vert f \Vert_{L^2_{\alpha}}=t^{\frac{1}{2}}\Vert f \Vert_{L^2_{v}}.
\]
On the Fourier side, we similarly obtain the pointwise bound
\[
 |(\hat w_{ell}, \hat r_{ell})| \lesssim (\xi^3 + t^\frac12 \xi^\frac34)^{-1}.
\]
These estimates allow us to place  $(w_{ell},r_{ell})$
into the error term in \eqref{pactest1}.  Further, in view of \eqref{hatq},
the contribution of $(w_{ell},r_{ell})$ to $\gamma$ has size $t^{-N}$,
and can be also placed in the error.

\bigskip 

{\bf (ii) The hyperbolic part: } Here spatial dyadic regions are diagonally matched
with dyadic frequency regions, and dyadic $l^2$ summability is inherited from the latter.
Hence it suffices to consider a fixed dyadic velocity range $\{v \approx v_0\}$, associated 
to the spatial region $\{|\alpha| \approx v_0t\}$,  and $(w,r)$ localized at frequency 
$\lambda_{v_0} \approx v_0^{-2}$.

 In order to fix signs, we first need to differentiate between the two symmetric cases
$v_0 > 0$ and $v_0 < 0$. Without any restriction in generality we take 
$v_0 > 0$. Then subtracting the two components in the second term
in \eqref{H} we obtain 
\[
\|(2\alpha \vert D\vert +t\vert D\vert^{\frac{1}{2}} )(w-r)\|_{L^2} \lesssim 1.
\]
The operator above is elliptic in $\{\alpha \approx v_0 t\}$, therefore 
we obtain
\[
\| \chi_{\alpha \approx v_0t} (w-r)\|_{L^2} \lesssim \frac{1}{t \lambda^{\frac12}},
\]
which is comparable to the estimates obtained in the elliptic case.
Thus, as there, we can bound $|D|^{s}(w-r)$ in $L^2$ and in $L^\infty$
and place it into the error term of \eqref{pactest1}.

Hence, we can freely replace $w$ and $r$ by $y = \frac{w+r}2$ in
\eqref{pactest1}.  We note that $\gamma$ is already expressed in terms
of $y$. To reduce the problem completely to an estimate for $y$ we
need one last step. Combining again the two components in the second term
in \eqref{H} we obtain 
\[
\| |D|^{\frac12} \chi_{\alpha \approx v_0 t} (4\alpha^2 \partial_{\alpha} + it^2)(w,r) \|_{L^2_{\alpha}} \lesssim t,
\]
which yields the same bound for $y$. Choosing  $\chi_{\alpha \approx vt}$ to be supported at 
spatial frequency $\ll v^{-2}$, we can cancel the $|D|^\frac12$ 
and conclude that 
\begin{equation}\label{z1}
\left\|  \chi_{\alpha \approx v_0 t} L y \right\|_{L^2_{\alpha}}
 \lesssim \frac{1}{v_0t}, \qquad L = \partial_{\alpha} +\frac{it^2}{4\alpha^2}.
\end{equation}
On the other hand, from the first relation in \eqref{H} we obtain
\begin{equation}\label{z2}
\| y \|_{L^2_\alpha} \lesssim (1+ v_0^{-2})^{-5}.
\end{equation}
From here on we will work only with the function $y$.

It is convenient to rewrite the bounds on $y$ in terms of the auxiliary function
$u:= e^{-i\phi} y$, which satisfies
$ \partial_{\alpha}u=e^{-i\phi}(\partial_{\alpha}+\frac{it^2}{4\alpha^2})y$.
Then for $u$ we have 
\begin{equation}\label{u1}
\left\|  \chi_{\alpha \approx v_0 t} \partial_\alpha u \right\|_{L^2_{\alpha}}
 \lesssim \frac{1}{v_0t}, \qquad \| u \|_{L^2_\alpha} \lesssim (1+ v_0^{-2})^{-5}.
\end{equation}
Combining these bounds we get by interpolation
\begin{equation*}
\label{ui}
\|  \chi_{\alpha \approx v_0 t} u \|_{L^\infty} \lesssim t^{-\frac12} v_0^{-\frac12}(1+ v_0^{-2})^{-\frac52},
\end{equation*}
which also is transferred back to $y$,
\begin{equation}\label{zi}
\|  \chi_{\alpha \approx v_0 t} y \|_{L^\infty} \lesssim t^{-\frac12} v_0^{-\frac12}(1+ v_0^{-2})^{-\frac52}.
\end{equation}

The bounds \eqref{z2} and \eqref{zi} lead directly to $L^2$ and $L^\infty$ bounds for $\gamma$,
\begin{equation}\label{game}
\|\gamma \|_{L^2_v(v \approx v_0)} \lesssim (1+ v_0^{-2})^{-5}, 
\qquad\|\gamma \|_{L^\infty(v \approx v_0)} \lesssim v_0^{-\frac12}(1+ v_0^{-2})^{-\frac52}.
\end{equation}
To estimate $\partial_v \gamma = \langle y, \partial_v \uu
\rangle_{L^2}$ we write $\partial_v \uu$ in the form
\[
\partial_v \uu = 
- v^{-\frac32} e^{i \phi} \left(  t \partial_\alpha \chi\left( \frac{\alpha-vt}{t^\frac12 v^{\frac32}}\right)
+ \frac32 \frac{\alpha-vt}{t^\frac12 v^\frac52} \chi'\left( \frac{\alpha-vt}{t^\frac12 v^{\frac32}}\right)
\right),
\]
and compute using integration by parts
\[
\partial_v \gamma = \int v^{-\frac32}  t \partial_\alpha u(t,\alpha) \chi\left( \frac{\alpha-vt}{t^\frac12 v^{\frac32}}\right) \, d\alpha -  \int v^{-\frac32}  u(t,\alpha)  \frac32 \frac{\alpha-vt}{t^\frac12 v^\frac52} \chi'\left( \frac{\alpha-vt}{t^\frac12 v^{\frac32}}\right) d\alpha .
\]
Now we can bound the two  integrals using \eqref{u1} to obtain
\[
\| \partial_v \gamma \|_{L^2_v(v \approx v_0)} \lesssim \frac{1}{v_0} ,
\]
which, together to \eqref{game}, concludes the proof of \eqref{gamma}.


It remains to estimate the $L^2$ and $L^{\infty}$ norms of the error
in \eqref{pacerr}.  We begin with the physical space error bounds.  The idea
is to bound the difference $$\err= y(t,vt) - t^{-\frac12}
e^{i\phi(t,vt)} \langle y, \uu \rangle_{L^2}$$ in both $L^2_v$ and
$L^\infty$ in terms of $\|y\|_{L^2_\alpha}$ and $\|Ly\|_{L^2_\alpha}$.
Precisely, we claim that
\begin{equation}
\label{err-est}
\| \err\|_{L^\infty(v \approx v_0)} \lesssim v_0^\frac34 t^{\frac{1}{4}}\Vert
 \chi_{\alpha \approx v_0 t} Ly \Vert_{L^2_{\alpha}}, \qquad \| \err\|_{L^2_v(v \approx v_0)} \lesssim 
v_0^{\frac32} \Vert
 \chi_{\alpha \approx v_0 t} Ly \Vert_{L^2_{\alpha}}.
 \end{equation}
Restated  in terms of $u$ we have $e^{-i\phi} \err  = y(t,vt)
- t^{-\frac12}  \langle u, e^{-i\phi} \uu \rangle_{L^2}$. Hence, using \eqref{chi-int},  we can write 
\begin{equation}
\label{diff}
\begin{aligned}
  e^{- i\phi} \err   &= \int  
(u(t,vt) -u(t,(v-z)t) )\chi \left(t^{\frac{1}{2}}v^{-\frac32} z \right) v^{-\frac32} t^{\frac{1}{2}}  \, dz.
\end{aligned}
\end{equation}
By H\"older's inequality
\begin{equation*}
\vert u(t,vt) -u(t,(v-z)t) \vert \leq 
\vert z \vert ^{\frac{1}{2}}  \Vert \partial_{v} u(t,vt)\Vert_{L^2_v}
=\vert z \vert ^{\frac{1}{2}} t^\frac12 \Vert \partial_{\alpha} u\Vert_{L^2_\alpha}.
\end{equation*}
As in \eqref{diff} $z$ is restricted to $|z| \lesssim v_0^\frac32 t^{-\frac12}$;
by \eqref{u1} we obtain 
\[
|\err| \lesssim v_0^{-\frac14} t^{-\frac34} .
\]
Hence the pointwise part of \eqref{err-est} follows.

To prove the $L^{2}_v$ bound in \eqref{err-est}, we estimate the RHS of
\eqref{diff} in terms of  $\partial_v u$,
\begin{equation*}
\begin{aligned}
\vert \err  \vert  &\lesssim \int_{0}^1 \int \vert z \vert \vert \partial_{v}u (t, (v-hz)t)\vert  \chi \left(t^{\frac{1}{2}}v^{-\frac32}z \right)v^{-\frac32} t^{\frac{1}{2}}  \, dz dh .
\end{aligned}
\end{equation*}
Thus, we can now evaluate the $L^2_v$ of the LHS of \eqref{diff} as follows,
\begin{equation*}
\begin{aligned}
\Vert \err  \Vert_{L^2_v(v \approx v_0)}  & \ \lesssim  \Vert \partial_{v} u(t, vt)\Vert_{L^2_v(v \approx v_0)}\int \vert z \vert  \chi \left(t^{\frac{1}{2}}v^{-\frac32} v_1 \right)v^{-\frac32} t^{\frac{1}{2}}  \, dz \\ & \ \approx t^{-\frac{1}{2}} v_0^\frac32  \Vert \partial_{v}u(t, vt)\Vert_{L^2_v(v \approx v_0)}
= v_0^{\frac32} \Vert \chi_{\alpha\approx v_0 t} \partial_{\alpha}u \Vert_{L^2_{\alpha}}.
\end{aligned}
\end{equation*}
This completes the proof of \eqref{err-est}. In turn, \eqref{err-est} together with \eqref{z1} applies directly to the case $s = 0$ of \eqref{pacerr} to give
\[
 \| \err_0\|_{L^\infty(v \approx v_0)} \lesssim v_0^{-\frac14} t^{-\frac{3}{4}},
\qquad \| \err_0 \|_{L^2_v(v \approx v_0)} \lesssim 
v_0^{\frac12} t^{-1} ,
 \]
which suffices for \eqref{pacerr1}.

In order to consider also the case  $s > 0$, we write
\[
\begin{split}
\err_s = & \  |D|^s y(t,vt) - (4v^2)^{-s} t^{-\frac12} e^{i\phi} \langle y, \uu\rangle_{L^2_\alpha} 
\\ = & \  |D|^s y(t,vt) -  
t^{-\frac12} e^{i\phi} \langle |D|^s y, \uu\rangle_{L^2_\alpha} 
-  t^{-\frac12} e^{i\phi} \langle y, (|D|^s -  (4v^2)^{-s}) \uu\rangle_{L^2_\alpha} 
\\ :=& \  \err_s^1 + \err_s^2.
\end{split}
\]
For $\err_s^1$ we apply \eqref{err-est} with $y$ replaced by $|D|^s y$, to obtain 
the same bound as before but with an added $v_0^{-2s}$ factor,
\[
 \|  \err_s^1\|_{L^\infty(v \approx v_0)} \lesssim v_0^{-\frac14} (1+v^{-2})^{-s} t^{-\frac{3}{4}},
\quad \|  \err_{s+\frac12}^1 \|_{L^2_v(v \approx v_0)} \lesssim 
v_0^{-\frac12} (1+v^{-2})^{-s}  t^{-1} ,
 \]
which is unfavorable  if $v_0 < 1$. But then we can still interpolate with the 
$t^{-\frac12}$ bound  from \eqref{zi} and \eqref{game} to remove the negative power of $v_0$.

For $\err_s^2$ instead we use the cancellation in \eqref{dsuu} to conclude that 
$P_\lambda (|D|^s -  (4v^2)^{-s}) \uu$ is a bump function comparable to $(4v^2)^{-s}
v^{\frac12} t^{-\frac12} \uu$. Hence we obtain a direct bound as in \eqref{game}, but 
with the same additional factor,
\[
 \| \err_s^2\|_{L^\infty(v \approx v_0)} \lesssim v_0^{-2s} (1+ v_0^{-2})^{-\frac52} t^{-1},
\qquad \| \err_s^1 \|_{L^2_v(v \approx v_0)} \lesssim 
v_0^{\frac12-2s}(1+ v_0^{-2})^{-5} t^{-1}, 
\]
which again suffices.

Finally, we consider the Fourier space error estimates. For this we
first need to switch the bounds for $y$ to the Fourier space. For a
fixed frequency $\lambda$ we have two dyadic regions in $v$ where the
hyperbolic components of $(w,r)$ are supported, namely those for which
$v_0^2 \approx \lambda$. Thus we will get two contributions in the
approximation to $(\hat w, \hat r)(t,\xi)$. As above, let us restrict
ourselves to the contribution corresponding to $v > 0$. Adding the two
components in the second term in \eqref{H} we obtain
\[
\| (2 \alpha |D| - t |D|^\frac12) y\|_{L^2} \lesssim 1.
\]
Since $y$ is localized at frequency $v_0^{-2}$, taking a Fourier transform
and estimating commutation errors via the first part of \eqref{H}, we obtain 
the main bounds for $\hat y$,
\begin{equation*}
\| \hat y\|_{L^2} \lesssim (1+v^{-2})^{-5}, \qquad \|  (\partial_\xi + i t |\xi|^{-\frac12}) \hat y\|_{L^2} 
\lesssim v_0^2.
\end{equation*}
On the other hand, the quantity to bound in $L^2_v$ and $L^\infty$
is 
\[
\hat{\err}_s = (\xi_v)^{s}\left (v^\frac32 \hat y(t,\xi_v) - 
e^{-it \xi_v} \langle \hat y,\hat \uu\rangle\right), \qquad \xi_v = -(4v^2)^{-1}.
\]
Given the expression of $\hat{\uu}$ in \eqref{hatq}, the error $\hat{\err}_s$ is estimated 
in the same fashion as in the proof of \eqref{err-est}.

\end{proof}

\subsection{The evolution of  $\gamma$}
Here we track the evolution of $\gamma(t,v)$ and prove
Proposition~\ref{p:ode}.  In view of the energy conservation relation for
the linear system \eqref{ww2d-0}, we directly obtain the relation
\[
\dot \gamma(t) = \int \tG \bar \ww + \tW \bar \ggg + i \tK_\alpha \bar \qq \, d\alpha  .
\]
We successively consider all terms on the right. With the exception of a single 
term, namely the resonant part of $G$, see below, all contributions will be placed 
into the error term $\sigma$. 
\bigskip

{\bf A. The contribution of $\bar \ggg$.}  This is 
\[
I_1 = v^{-\frac32} \int \tW e^{-i\phi} \left( \partial_{\alpha}\left[ \frac{(\alpha-vt) }{2\alpha}\chi-i \frac{ (\alpha +vt)^2}{4v^{\frac{3}{2}}t^{\frac{5}{2}}}\chi '\right]+  \left[ \frac{(\alpha-vt)}{2\alpha^2}\chi -i \frac{ (\alpha -vt)}{4v^{\frac{3}{2}}t^{\frac{5}{2}}}\chi'\right]\right) \, d\alpha.
\]
We use \eqref{pactest} to replace $W$ in terms of $\gamma$, and the
contribution of the error $\tW - t^{-\frac12}e^{i\phi} \gamma(t,v)$ is
directly estimated in both $L^2$ and $L^\infty$ via \eqref{pacerr}.

The contribution of $\gamma$, on the other hand, is written using integration by parts as 
\[
\tilde I_1 = v^{-\frac32} t^{-\frac12} \int  - \gamma_\alpha \left[ \frac{(\alpha-vt) }{2\alpha}\chi-i \frac{ (\alpha +vt)^2}{4v^{\frac{3}{2}}t^{\frac{5}{2}}}\chi '\right]+  \gamma \left[ \frac{(\alpha-vt)}{2\alpha^2}\chi -i \frac{ (\alpha -vt)}{4v^{\frac{3}{2}}t^{\frac{5}{2}}}\chi'\right] \, d\alpha.
\]
Now we can easily bound the two terms using \eqref{gamma}. 

\bigskip

{\bf B. The contribution of $\tG$ and $\tK$.}  
For this we consider in more detail the structure of $\tG$ and
$\tK$. We will successively peel off favorable terms until we are left 
only with the leading resonant part.

\medskip
{\bf B1. Quartic and higher order terms.}
We decompose them into cubic and higher terms,
\[
\tG = \tG^{(3)} + \tG^{(4+)}, \qquad \tK = \tK^{(3)} + \tK^{(4+)}.
\]
In view of \eqref{GK4},  we can estimate the contribution of the quartic and higher
terms in $L^\infty$,  
\[
\left| \int   \tG^{4+} \bar \ww +  i \tK^{4+}_\alpha \bar \qq \, d\alpha\right|
\lesssim \epsilon^4 t^{-\frac32+3C^2_{*} \epsilon^2}    \| (\ww,\qq)\|_{\dH_0}
\lesssim   \epsilon^4 v^{-\frac12} t^{-\frac54+3C^2_{*} \epsilon^2},    
\]
which suffices in $\Omega$. The $L^2$ bound is similar, using again \eqref{GK4}.

In the same way, by using \eqref{GK-diff}, we can estimate the contribution
of the difference 
\[
 (\tilde G^{3}(W,R), \tilde K^{3}(W,R)) - (\tilde
G^{3}(\tW,\tQ_\alpha), \tilde K^{3}(\tW,\tQ_\alpha)),
\]
which also contains only quartic and higher terms. 
Thus we have substituted $(\tG,\tK)$ with the cubic expressions
$(\tilde G^{3}(\tW,\tQ_\alpha), \tilde K^{3}(\tW,\tQ_\alpha))$. 
\medskip

{\bf B2. Cubic terms.}  To better understand the cubic interactions we
need the folowing heuristic analysis:

\begin{itemize}
\item[(i)] in the physical space, waves at frequency $\xi$ move with
  velocity $\pm \frac12 |\xi|^{-\frac12}$. Since our data is localized
  near the origin, it follows that the bulk of the solution at
  $(\alpha,t)$ is at space-time frequency $(\xi,\tau) =
  \left(-\dfrac{t^2}{4\alpha^2}, \dfrac{t}{2\alpha}\right)$. Thus the
  worst cubic interactions are those of waves with equal frequency.

\item[(ii)] In the frequency space, trilinear interactions of equal frequency waves 
$(\xi,\pm \sqrt{ |\xi|})$ (with $ \xi < 0$) can only lead back to the characteristic set if 
exactly one complex conjugation is present.  
\end{itemize}
This leads to the following classification of the terms in $(\tilde
G^{3}(\tW,\tQ_\alpha), \tilde K^{3}(\tW,\tQ_\alpha))$:

\begin{itemize}
\item[A.] Nonresonant trilinear terms: these are either 
(A1) terms with no complex conjugates, or  
(A2) terms with two complex conjugates.

\item[B.] Resonant trilinear terms: terms with exactly one
  conjugation.  For such terms one may further define a notion of
  principal symbol, which is the leading coefficient in the expression
  obtained by substituting the factors in the trilinear form by the
  expressions in \eqref{pactest} \footnote{Which corresponds to all
    three frequencies being equal.}  Thus one can isolate a linear subspace
  of resonant terms for which this symbol vanishes, which we call {\em
    null terms}.  Hence on the full class of resonant trilinear terms
  we can further define an equivalence relation, modulo null terms.
\end{itemize}

Now we turn our attention to the situation at hand, where we recall that
\begin{equation*}
\left\{
\begin{aligned}
\tilde G^{(3)}(W,R) = & \  2P[ -\partial_{\alpha}P\left[ R\bar{ \W}-\bar{R}\W\right] \Re W +(R \W)_{ \alpha}  \Re W + \W \Re (\W R)]   \\ & \ -  P[\bar \W^2  R- \W(P[\bar R \W]+\bar P[R \bar \W]) ]\\
\tilde K^{(3)}(W,R) = &\  2P\left[ P [(R+\bar R)R_\alpha]\Re W   +     i \W^2 \Re W + P[R \bar R_\alpha] \Re W \right] \\ & \ +P\left[\bar{R}\bar{\W}R -R\bar{P}\left[ \bar{R}\W-R\bar{\W} \right]\right].
\end{aligned}
\right.
\end{equation*}
Based on the previous heuristics, we decompose
\begin{equation*}
\left\{
\begin{aligned}
\tilde G^{(3)} =&\tG^{(3)}_{r}+\tG^{(3)}_{nr}+\tG^{(3)}_{null}\\
\tilde K^{(3)} = & \tK^{(3)}_{r}+\tK^{(3)}_{nr}+\tK^{(3)}_{null},
\end{aligned}
\right.
\end{equation*}
where
\begin{equation*}
\left\{
\begin{aligned}
\tG^{(3)}_{r}=&P[(R \W)_{ \alpha}  \bar{ W} +\W R\bar{\W}]\\
\tG^{(3)}_{nr}=&P[(R \W)_{ \alpha} W  +\W \W R]\\
\tG^{(3)}_{null}=&P[ -2\partial_{\alpha}P[ R\bar{ \W}-\bar{R}\W ] \Re W +  \bar{\W} (\W\bar{ R}-\bar \W  R)+\W P[\bar{R}\W-R\bar{\W}] ]\\
\tilde K^{(3)} _{r}= &0\\
\tilde K^{(3)} _{nr}=&P\left[\bar{R}\bar{\W}R\right]  \\
\tilde K^{(3)} _{null}=&P[ -R\bar{P}\left[ \bar{R}\W-R\bar{\W} \right]+2  P[|R|^2]_\alpha \Re W + 2(RR_{\alpha}+i\W^2)\Re W] .
\end{aligned}
\right.
\end{equation*}
We will place all cubic contributions into  the error term $\sigma$,  
except for  the contribution of the resonant part  $\tG_{r}$.  

We note that for the most part the
exact form of the expressions above is irrelevant.  The only
significant matter is the coefficient of the terms in $\tG^{(3)}_{r}$,
which needs to be real\footnote{ A similar constraint would be
  required of the coefficients in $\tK^{(3)}_{r}$, if they were
  nonzero.}.

We also remark that the leading projection in all terms can be harmlessly discarded, 
since it can be moved onto the wave packets, which decay rapidly at positive frequencies,
\[
\| (\ww,\qq) - P (\ww,\qq) \|_{H^N} \lesssim t^{-N}.
\]

\medskip

{\bf B2(a). Substitution by the asymptotic expansion.}
The first step in our estimates for $(\tilde G^{3}(\tW,\tQ_\alpha),
\tilde K^{3}(\tW,\tQ_\alpha))$ is to show that we can harmlessly replace 
$(\tW,\tQ_\alpha)$ by their leading asymptotic expression in \eqref{pactest}.
Denoting the resulting expressions by $(\tilde G^{3}(\gamma),
\tilde K^{3}(\gamma))$, we claim that we can place the expression
\[
\langle (\tilde G^{3}(\tW,\tQ_\alpha) - \tilde G^{3}(\gamma), \tilde K^{3}(\tW,\tQ_\alpha))
- \tilde K^{3}(\gamma)), (\ww,\qq)\rangle_{\dot H_0}
\]
into the error term $\sigma$.

We first consider terms without inner projections. To fix the notations, 
consider the expression  $\tilde G^{(3)}_{r}$, with the projection $P$ removed.
Expanding the derivatives, we have 
\[
\tilde G^{3}_r(\tW,\tQ_\alpha) = \tQ_{\alpha\alpha} \tW_\alpha \bar \tW
+ \tQ_{\alpha} \tW_{\alpha\alpha} \bar \tW +\tQ_{\alpha\alpha} \tW_\alpha \bar \tW_\alpha.
\]
On the other hand, a direct computation, using  \eqref{pactest}, shows that
\begin{equation}\label{g3r}
\tilde G^{3}_r(\gamma) = \frac{i}{t^\frac32} \left(\frac{1}{2v}\right)^5 
e^{i\phi} \gamma |\gamma|^2 .
\end{equation}
Within the region of interest $\Omega$, the difference 
$\tilde G^{3}_r(\tW,\tQ_\alpha) - \tilde G^{3}_r(\gamma)$ is a sum of cubic expressions
with either one $\gamma$ factor and two error factors, or viceversa. 
Using \eqref{gamma} for $\gamma$   and \eqref{pacerr} for the errors, 
we immediately obtain the bound
\[
\| (1+v^{-2})^2 \tilde G^{3}_r(\tW,\tQ_\alpha) - \tilde G^{3}_r(\gamma)\|_{L^2 \cap L^\infty}
\lesssim \epsilon^3 t^{-\frac32- \frac{1}{18}+ C^2_{*} \epsilon^2} ,
\]
where both norms are estimated in $\Omega$. Matching this against $\ww$ we 
can directly place the output in the error $\sigma$.

The same strategy works for terms with an inner projection $P$, with two additional 
observations:

(i) Since $P$ is nonlocal, one also needs to consider contributions
from outside $\Omega$.  But this is straightforward due to the better
decay estimate in \eqref{first-point-tilde} for $v$ away from $1$.

(ii) Since $P$ is not bounded in  $L^\infty$, there is an additional logarithmic  loss
in the $L^\infty$ bound for the error $\sigma$.

\medskip
 
{\bf B2(b) The contribution of the null terms $\tG^{(3)}_{null}(\gamma)$ and
  $\tK^{(3)}_{null}(\gamma)$.}  Here we simply note that
$\tG^{(3)}_{null}(\gamma)= 0$ and $\tK^{(3)}_{null}(\gamma)= 0$, so
after the previous step there is nothing left to do.  We
remark that cancellation actually occurs at the bilinear level for the
``null expressions'' of type
\[
\tW_\alpha \bar \tQ_\alpha - \bar \tW_\alpha  \tQ_\alpha, \qquad 
(|\bar \tQ|^2)_\alpha, \qquad \tQ_\alpha \tQ_{\alpha \alpha}+i\tW_\alpha^2.
\]

\medskip

{\bf  B2(c) The contribution of the nonresonant terms
$\tG^{(3)}_{nr}$ and $\tK^{(3)}_{nr}$.}
Here it is important that we integrate against $\ww$ and $\qq$, as
that fixes the frequency of the output at $\xi = -
\frac{1}{4v^2}$. On the other hand the nonresonant trilinear
expression will be concentrated at frequency $3 \xi$ if no complex
conjugate occur, respectively at frequency $-\xi$ if two conjugates
occur. The easiest way to take advantage of this mismatch is via an
integration by parts argument. Consider for instance the expression in 
$\tG^{(3)}_{nr}(\gamma)$, which has the form
\[
\tG^{(3)}_{nr}(\gamma) =  - \frac{3i}{t^\frac32} \left(\frac{1}{2v}\right)^5 
e^{3 i\phi} \gamma^3 .
\]
Testing this against $\ww$, we can write
\[
\int \tG^{(3)}_{nr}(\gamma) \bar \ww \, d\alpha = 
\int  - \frac{3i}{t^\frac32} \left(\frac{t}{2\alpha}\right)^5 e^{3 i\phi} \gamma^3 \ww \, d\alpha
= \int  - \frac{3i}{t^\frac32} \left(\frac{t}{2\alpha}\right)^5 e^{2 i\phi} \gamma^3 
\tilde \chi\left(\frac{\alpha - vt}{t^\frac12 v^\frac32}\right)
v^{-\frac32}  \, d\alpha.
\]
Here the phase $2\phi$ is nonstationary, so we can integrate by parts to place 
a derivative on either  $\tilde \chi$ or on  $\gamma$. In the first case we gain 
a $t^{-\frac12}$ factor directly, while in the second case a similar gain comes from 
\eqref{gamma}.

\medskip

{\bf B2(d) The contribution of the resonant 
term $\tG^{(3)}_{r}(\gamma)$.} 
Given the expression \eqref{g3r}, all we need is to consider the 
integral
\[
I  = \int \frac{i}{t^\frac32} \left(\frac{t}{2\alpha}\right)^5 
 \gamma(t.\alpha/t)  |\gamma(t,\alpha/t)|^2  e^{-i\phi} \bar \ww \, d\alpha.
\]
Here $e^{-i\phi} \bar \ww$ has the form
\[
e^{-i\phi} \bar \ww = \frac12 v^{-\frac32} \chi\left(\frac{\alpha - vt}{t^\frac12 v^\frac32}\right)
+ v^{-1} t^{-\frac12} 
\tilde \chi\left(\frac{\alpha - vt}{t^\frac12 v^\frac32}\right)
\]
with Schwartz functions $\chi$ and $\tilde \chi$ so that $\int \chi = 1$.
We can freeze the coefficient $\frac{t}{2\alpha}$ at $\alpha = vt$ at the expense of
modifying $\tilde \chi$ to write 
\[
I = \frac{i}{t^\frac32} (2v)^{-5} \int \gamma(t.\alpha/t)
|\gamma(t,\alpha/t)|^2 \left[  \frac12 v^{-\frac32} \chi\left(\frac{\alpha - vt}{t^\frac12 v^\frac32}\right)
+ v^{-1} t^{-\frac12} 
\tilde \chi\left(\frac{\alpha - vt}{t^\frac12 v^\frac32}\right)\right] \, d\alpha := J + \tilde J.
\]
The contribution $\tilde J$ of the lower order term containing $\tilde
\chi$ is part of the error $\sigma$, and is estimated directly in
$L^2$ and $L^\infty$ using the $L^2$ and $L^\infty$ bounds for
$|\gamma|^3$ derived from \eqref{gamma}.

For the contribution $J$ of the main term containing $\chi$ we freeze $\gamma$
at the packet center to obtain the leading contribution:
\[
J = \frac{i}{2 t} (2v)^{-5} \gamma(t,v) |\gamma(t,v)|^2 + J_1,
\]
where
\[
J_1 = \frac{i}{2 t^\frac32}(2v)^{-5}
\int \left[ \gamma(t.\alpha/t)
|\gamma(t,\alpha/t)|^2-    \gamma(t.v)
|\gamma(t,v)|^2\right] v^{-\frac32} \chi\left(\frac{\alpha - vt}{t^\frac12 v^\frac32}\right) \, d\alpha .
\]
We still need to bound the last integral in $L^\infty$ and in $L^2_v$. For  
this we first use the estimates in  \eqref{gamma} to conclude that 
\[
\| v^2(1+v^{-2})^5  \partial_v [\gamma(t,v) |\gamma(t,v)|^2]\|_{L^2_v} \lesssim t^{C^2_{*}\epsilon^2} . 
\]
Then, by the same argument as in the proof of \eqref{err-est}, we obtain
\[
\| (1+v^{-2})J_1\|_{L^\infty} \lesssim t^{-\frac54+ C^2_{*}\epsilon^2} , \qquad 
 \| J_1\|_{L^2_v} \lesssim t^{-\frac32+ C^2_{*}\epsilon^2},
\]
which are more than sufficient in order to include $J_1$ in the error term $\sigma$.
Thus the proof of Proposition~\ref{p:ode} is concluded.

\subsection{The asymptotic expansion of the solution}

To construct the asymptotic profile $\Psi$ we use the differential equation in  Proposition~\ref{p:ode}
for $\gamma(t,v)$. The inhomogeneous term $\sigma(t,v)$ is estimated in $L^\infty$ 
and $L^2_v$ as showed in \eqref{sigma}.
The differential equation for $\gamma$ in \eqref{gamma-ode} can be explicitly solved in polar
coordinates. Since $\sigma(t,v)$ in uniformly integrable in time, it
follows that for each $v$, $\gamma(t,v)$ is well approximated at infinity by a
solution to the unperturbed differential equation,  
in the sense that
\begin{equation}
\label{daniel}
\gamma (t,v)=\Psi\left(v \right) e^{\frac{i}2 (2v)^{-5} \ln{t} |\Psi(v)|^2}+O_{L^{\infty}}(\epsilon t^{-\frac{1}{18}+C_{*}^2\epsilon^2}).
\end{equation}
Integrating the $L^2_v$ part of \eqref{sigma}  leads to a similar $L^2_v$ bound 
\begin{equation}
\label{daniel2}
\gamma (t,v)=\Psi\left(v \right)
e^{\frac{i}2 (2v)^{-5} \ln{t} |\Psi(v)|^2}+O_{L_v^{2}}(\epsilon t^{-\frac{1}{18}+C_{*}^2\epsilon^2}).
\end{equation}
Both of these relations are valid in $\Omega$, where we know that \eqref{gamma} holds.
Then from the two relations in \eqref{pactest1} we obtain the asymptotic expansions in \eqref{asymptotics} within $\Omega$ but for the normal form variables $(\tilde{W}, \tilde{Q})$. The transition to the original variables $(W,Q)$ is straightforward in view of the expressions \eqref{nft1}, and the pointwise decay bounds for $(W, Q)$.

The next step is to establish the regularity of $\Psi$.
On one hand, from \eqref{daniel} and \eqref{daniel2} we get (within $\Omega$)
\[
\| \Psi(v) - \gamma(t,v) e^{-\frac{i}2 (2v)^{-5} |\gamma(t,v)|^2 \log t}\|_{L^2_v\cap L^{\infty}} \lesssim  \epsilon t^{-\frac{1}{18}+C_{*}^2 \epsilon^2} \log t ,
\]
while by \eqref{gamma} we have the $L^2_{v}$ bound 
\[
\Vert (1+v^{-2})^5 \gamma(t,v) e^{-\frac{i}2 (2v)^{-5} |\gamma(t,v)|^2 \log t} \|_{L^2_v} + \| v\partial_v [\gamma(t,v) e^{-\frac{i}2 (2v)^{-5} |\gamma(t,v)|^2 \log t}] \|_{L^2_v} \lesssim \epsilon t^{C_{*}^2 \epsilon^2}  \log t ,
\]
and the $L^{\infty} $ bound
\[
\Vert v^{\frac{1}{2}}(1+v^{-2})^{\frac{5}{2}} \gamma(t,v) e^{-\frac{i}2 (2v)^{-5} |\gamma(t,v)|^2 \log t} \|_{L^{\infty}}  \lesssim \epsilon t^{C_{*}^2 \epsilon^2}  .
\]
As $t$ increases, the range $[t^{-\frac19}, t^{\frac19}]$ of $|v|$ within $\Omega$ increases to cover the entire $\mathbb{R}^+$. 
Hence, by interpolation we obtain that for large enough $C_{*}$, $\Psi$ has the regularity 
\begin{equation}
\label{interpolation}
\Vert \vert v\vert ^{\frac{1}{2}-C_{*}^2\epsilon^2}(1+v^{-2})^{\frac{5}{2}-C_{*}^2\epsilon^2} \Psi \Vert  _{L^{\infty}}\lesssim \epsilon ,\quad 
\Vert (1+v^{-2})^{5-C_{*}^2\epsilon^2} \Psi \Vert  _{L^{2}_v}\lesssim \epsilon ,
\end{equation}
\[
\| \vert v\vert ^{1-C_{*}^2\epsilon^2}\Psi \|_{H^{1-C_{*}^2\epsilon^2}_v} \lesssim \epsilon.
\]

Finally, we can combine the pointwise bounds for $\Psi$ with the pointwise bounds for $(W,Q)$ to extend the error estimates in \eqref{asymptotics} to the exterior of 
$\Omega$.

\end{document}